\documentclass{amsart}
\usepackage{graphicx}
\usepackage{amssymb}

\newtheorem{thm}{Theorem}[section]

\newtheorem{lem}[thm]{Lemma}

\theoremstyle{definition}

\theoremstyle{remark}
\newtheorem{rem}[thm]{Remark}

\newcommand{\del}{\partial}
\newcommand{\dbar}{\overline{\del}}

\begin{document}

\title[Dimension of the minimum set]{Dimension of the minimum set for the real and complex Monge-Amp\`{e}re equations in critical Sobolev spaces}
\author[T. Collins]{Tristan C. Collins$^{*}$}
\address{Department of Mathematics, Harvard University, 1 Oxford Street, Cambridge, MA 02138}
\email{tcollins@math.harvard.edu}
\author[C. Mooney]{Connor Mooney$^{**}$}
\address{Department of Mathematics, ETH Z\"{u}rich, R\"{a}mistrasse 101, 8092 Z\"{u}rich, Switzerland}
\email{\tt connor.mooney@math.ethz.ch}
\thanks{$^{*}$Supported in part by National Science Foundation grant DMS-1506652, the European Research Council and the Knut and Alice Wallenberg Foundation}
\thanks{$^{**}$Supported in part by National Science Foundation fellowship DMS-1501152 and by the ERC grant ``Regularity and Stability in Partial Differential Equations (RSPDE)''}


\begin{abstract}
We prove that the zero set of a nonnegative plurisubharmonic function that solves $\det (\partial \dbar u) \geq 1$ in $\mathbb{C}^n$ and is in $W^{2, \frac{n(n-k)}{k}}$ contains no analytic sub-variety of
dimension $k$ or larger. Along the way we prove an analogous result for the real Monge-Amp\`{e}re equation, which is also new.
These results are sharp in view of well-known examples of Pogorelov and B\l ocki.
As an application, in the real case we extend interior regularity results to the case that $u$ lies in a critical Sobolev space (or more generally, certain Sobolev-Orlicz spaces).
\end{abstract}

\maketitle

\section{Introduction}
In this paper we investigate the dimension of the singular set for the real and complex Monge-Amp\`{e}re equations, assuming critical Sobolev regularity.

We first discuss the real case. It is well-known that convex (e.g. viscosity) solutions to $\det D^2u = 1$ are not always classical solutions. Pogorelov constructed examples in dimension $n \geq 3$ of the form
$$u(x',\,x_n) = |x'|^{2-\frac{2}{n}}f(x_n),$$
that solve $\det D^2u = 1$ in $|x_n| < \rho$ for some $\rho > 0$ and some smooth, positive $f$. This example is $C^{1,\,\alpha}$ for $\alpha \leq 2 - \frac{2}{n}$, and $W^{2,\,p}$ for $p < \frac{n(n-1)}{2}$.
Furthermore, this solution is not strictly convex, and its graph contains the line segment $\{x' = 0\}$.

On the other hand, it is known that strictly convex solutions are smooth. The proof of this fact is closely related to the solution of the Dirichlet problem, which has a long history, 
beginning with work of Pogorelov \cite{Pog3,Pog,Pog1, Pog2}, Cheng-Yau \cite{CY, CY1} and Calabi \cite{Cal}.  
Cheng-Yau solved the Minkowski problem on the sphere \cite{CY1}, and proved the existence of solutions to the Dirichlet problem which are smooth in the interior and Lipschitz up to the boundary \cite{CY}.  P.L. Lions
gave an independent proof of this result \cite{L1,L}.   Caffarelli-Nirenberg-Spruck \cite{CNS} and Krylov \cite{Kr2} established the existence of solutions smooth up to the boundary, provided the boundary data are $C^{3,\,1}$.  
Trudinger-Wang proved optimal boundary regularity results \cite{TW}, where the optimality comes from earlier examples of Wang \cite{W}.

\begin{rem}
In the case $n = 2$ it is a classical result of Alexandrov that solutions to $\det D^2u \geq 1$ are strictly convex \cite{A}.
\end{rem}

In view of the above discussion, to show interior regularity for $\det D^2u = 1$ it is enough to show strict convexity. 
(We remark that interior {\it estimates} generally depend on the modulus of strict convexity). Urbas \cite{U} showed strict convexity when $u$ is in $C^{1,\,\alpha}$ for $\alpha > 1 - 2/n$, or
in $W^{2,\,p}$ for $p > \frac{n(n-1)}{2}$. (Note that for these values of $p$, $W^{2,\,p}$ embeds into $C^{1,\alpha}$ for $\alpha > 1-\frac{2}{n-1}$, so neither result implies the other). 
Caffarelli showed that if $1 \leq \det D^2 u < \Lambda$ and $u$ is not strictly convex, then the graph of $u$ contains an affine set with no interior extremal points \cite{Ca}, and if 
$\det D^2u \geq 1$, then the dimension of any affine set in the graph of $u$ is strictly smaller than $\frac{n}{2}$ (\cite{Ca3}, see also \cite{M}).
These results led to interior $C^{2,\, \alpha}$ and $W^{2,\,q}$ estimates for solutions with linear boundary data, when $\det D^2u$ is strictly positive and $C^{\alpha}$, resp. $C^0$ (\cite{Ca1}).
Finally, in \cite{M} the second author showed that if $\det D^2u \geq 1$, then $u$ is strictly convex away from a set of Hausdorff $n-1$ dimensional measure zero, 
and that this is optimal by example (even when $\det D^2u = 1$).

In view of the Pogorelov example, the $C^{1,\alpha}$ hypothesis in \cite{U} is sharp, and the $W^{2,\,p}$ hypothesis is nearly sharp. 
In this paper we show interior regularity for the borderline case $p = \frac{n(n-1)}{2}$. Our result in the real case is:

\begin{thm}\label{RealMain}
Assume that $u$ is a convex solution to $\det D^2u \geq 1$ in $B_1 \subset \mathbb{R}^n$, and let $0 < k < \frac{n}{2}$.
If $u \in W^{2,\, p}(B_1)$ for some $p \geq \frac{n}{2k}(n-k)$, then the dimension of the set where $u$ agrees with a tangent plane
is at most $k - 1$.
\end{thm}

\begin{rem}
In particular, if $u \in W^{2,\, \frac{n(n-1)}{2}}$, then it is strictly convex.
We in fact show that $u$ is strictly convex if $\Delta u$ lies in Orlicz spaces that are slightly weaker than $L^{\frac{n(n-1)}{2}}$ (see Section \ref{RealMainProof}), strengthening
the result from \cite{U}. Our result is sharp in view of the Pogorelov example.

As a consequence, we can extend interior estimates to the borderline case $p = \frac{n(n-1)}{2}$ (see Section \ref{Applications}).
Interior estimates of this kind are often important in geometric applications, where one does not control the boundary data.
\end{rem}

\begin{rem}
There are analogues of the Pogorelov example that vanish on sets of dimension $k$ for any $k < \frac{n}{2}$, and are not in $W^{2,\, \frac{n}{2k}(n-k)}$ (\cite{Ca3}). 
These show that Theorem \ref{RealMain} is also sharp in the case $k > 1$.
\end{rem}

We now discuss the complex case. Like in the real case, there exist singular Pogorelov-type examples of the form
$$u(z', \,z_n) = |z'|^{2-\frac{2}{n}}f(z_n)$$
for $n \geq 2$, such that $\det (\partial \dbar u)$ is a strictly positive polynomial (\cite{B}).

\begin{rem}
In fact, there are analogues of this example that vanish on sets of complex dimension $k$ for any $k < n$. Furthermore, these singular examples are global.
\end{rem}

Less is known about interior regularity for the complex Monge-Amp\`{e}re equation $\det (\partial \dbar u) = 1$. B\l ocki and Dinew \cite{BD} showed that if $u \in W^{2,\,p}$ for some $p > n(n-1)$,
then $u$ is smooth. This result relies on an important estimate of Ko\l odziej \cite{Ko}. The same result is true provided $\Delta u$ is bounded (see e.g. \cite{Wa}). In this case the point is that the operator becomes
uniformly elliptic, and by its concavity an important $C^{2,\alpha}$ estimate of Evans and Krylov (see e.g. \cite{CC}) applies. Thus far, there does not seem to be a geometric condition analogous to strict convexity that guarantees
interior regularity.

However, if $u$ is nonnegative then something can be said about analytic structures in the minimum set.  A classical theorem of Harvey and Wells \cite{HW} says that the minimum set of a smooth, strictly plurisubharmonic function is contained in a $C^1$, totally real submanifold. Dinew and Dinew \cite{DD} recently showed that if $\det (\partial \dbar u)$ has a positive lower bound and $u \in C^{1,\,\alpha}$ for $\alpha > 1 - \frac{2k}{n}$ (or $C^{\beta}$ for $\beta > 2-\frac{2k}{n}$ in the case $k > \frac{n}{2}$),
then the minimum set of $u$ contains no analytic sub-varieties of dimension $k$ or larger. We investigate the same situation assuming Sobolev regularity.
In the complex case, our main result is:

\begin{thm}\label{ComplexMain}
Assume that $u$ is a nonnegative plurisubharmonic function satisfying $\det (\partial \dbar u) \geq 1$ in $B_1 \subset \mathbb{C}^n$, and let $0 < k < n$. If $\Delta u \in L^p$ for some $p \geq \frac{n}{k}(n-k)$, then the zero set $\{u = 0\}$ contains
no analytic sub-varieties of dimension $k$ or larger.
\end{thm}

\begin{rem}
Since $W^{2, \frac{n}{k}(n-k)}$ embeds into $C^{1,\, 1-\frac{2k}{n-k}}$, this result is different from that in \cite{DD}.
It is sharp in view of Pogorelov-type examples.
\end{rem}

\begin{rem}
It is not known whether all singularities of solutions to $\det (\partial \dbar u) = 1$ arise as analytic sub-varieties, or that they occur on a complex analogue of the agreement set with a tangent plane.
Thus, Theorem \ref{ComplexMain} does not immediately imply smoothness of solutions to $\det (\partial \dbar u) = 1$ when $u \in W^{2,\, n(n-1)}$ (unlike in the real case).
\end{rem}

The critical Sobolev spaces arise naturally in geometric applications.  For example, in complex dimension $2$ the $L^2$ norm of the Laplacian is a scale invariant, monotone quantity whose concentration controls, at least qualitatively, the regularity of functions with Monge-Amp\`ere mass bounded below.  In this sense, Theorem~\ref{ComplexMain} can be seen as a step toward understanding the regularity and compactness properties of sequences of (quasi)-PSH functions with lower bounds for the Monge-Amp\`ere mass, which arise frequently in K\"ahler geometry.

The proof of Theorem \ref{RealMain} relies on two key observations. The first is that $u$ grows at least like $\text{dist.}^{2-\frac{2k}{n}}$ away from a zero set of dimension $k$. The second is that the $W^{2,\frac{n}{2k}(n-k)}$ norm is invariant
under the rescalings that fix the $k$-dimensional zero set, and preserve functions with this growth. By combining these observations with some convex analysis, 
we show that the mass of $(\Delta u)^{\frac{n}{2k}(n-k)}$ is at least some fixed positive constant in each dyadic annulus around the zero set.

In the complex case the strategy is similar, but an important difficulty is that we don't have convexity. We overcome this in two ways. First, using subharmonicity along complex lines we can say that $u$ grows at a certain rate from its zero 
set at many points. Second, we use a dichotomy argument: either the mass of $|D^2u|^{\frac{n(n-k)}{k}}$ is at least a small constant in an annulus around the zero set, or it is very large and concentrates close to the zero set. 
Using that the $W^{2,\frac{n(n-k)}{k}}$ norm is bounded, we can rule out the second case and proceed as before.

The paper is organized as follows. In Section \ref{Preliminaries} we prove some estimates from convex analysis that are useful in the real case. We then prove an analogue in the general setting that is useful in the complex
case. In Section \ref{RealMainProof} we prove Theorem \ref{RealMain}. In Section \ref{ComplexMainProof} we prove Theorem \ref{ComplexMain}. Finally, in Section \ref{Applications} we give some applications of
Theorem \ref{RealMain} to interior estimates for the real Monge-Amp\`{e}re equation.


\section{Preliminaries}\label{Preliminaries}
Here we prove some useful functional inequalities.
The first inequality is from convex analysis. This will be used to prove Theorem \ref{RealMain}. 
We then prove a certain analogue in the general setting. This will be used to prove Theorem \ref{ComplexMain}.

\subsection{Estimate from Convex Analysis}

\begin{lem}\label{AnnulusMass}
Let $n \geq 2$ and let $w$ be a nonnegative convex function on $B_2 \subset \mathbb{R}^n$, with $w(0) = 0$ and $\sup_{\partial B_1} w \geq 1$. Then there is some positive dimensional constant $c(n)$ such that
$$\int_{B_2 \backslash B_1} \Delta w \,dx > c(n).$$
\end{lem}

\begin{proof}
By integration by parts, we have
$$\int_{B_2 \backslash B_1} \Delta w\,dx = \int_{\partial B_2} \partial_rw \,ds - \int_{\partial B_1} \partial_rw \,ds,$$
where $\partial_r$ denotes radial derivative.
By convexity, $\partial_r w$ is increasing on radial lines. We conclude that
$$\int_{B_2 \backslash B_1} \Delta w\,dx \geq \frac{1}{2} \int_{\partial B_2} \partial_r w\,ds.$$
Assume that the maximum of $w$ on $\partial B_1$ is achieved at $e_n$. By convexity, $w \geq 1$ in $B_2 \cap \{x_n \geq 1\}$, hence $\partial_rw > \frac{1}{2}$ on $\partial B_2 \cap \{x_n \geq 1\}.$
Since $\partial_r w \geq 0$, the conclusion follows.
\end{proof}

As a consequence, the Sobolev regularity of a convex function whose maximum on $\partial B_r$ grows like $r^q$ is no better than that of $r^q$:

\begin{lem}\label{CriticalBlowup}
Assume that $w$ is a nonnegative convex function on $B_1 \subset \mathbb{R}^n$ ($n \geq 2$), such that $w(0) = 0$ and $\sup_{\partial B_r} w \geq r^q$ for some $q \in [1,\,2)$, and all $r < 1$. Then
$$\int_{B_1 \backslash B_r} (\Delta w)^{\frac{n}{2-q}}\,dx \geq c(n,\,q) |\log r|$$
for some $c(n,\,q) > 0$ and all $r \in \left(0,\,\frac{1}{2}\right)$.
\end{lem}

\begin{rem}
We take $q \geq 1$ since convex functions are locally Lipschitz.
\end{rem}

\begin{proof}
Fix $\rho < 1/2$ and let $w_{\rho}(x) = {\rho}^{-q}w(\rho x)$. Note that the $L^{\frac{n}{2-q}}$ norm of $\Delta w$ is invariant under such rescalings. We conclude from this observation and Lemma \ref{AnnulusMass} that
$$\int_{B_{2\rho} \backslash B_{\rho}} (\Delta w)^{\frac{n}{2-q}} \,dx = \int_{B_2 \backslash B_1} (\Delta w_{\rho})^{\frac{n}{2-q}}\,dx \geq c(n,\,q).$$
The estimate follows by summing this inequality over dyadic annuli.
\end{proof}

\begin{rem}\label{OrliczGeneral}
One can refine this estimate to Orlicz norms. Let $F: [0,\, \infty) \rightarrow [0,\, \infty)$ be a convex function with $F(0) = 0$.
By Lemma \ref{AnnulusMass} we have 
$$\frac{1}{|B_{2r} \backslash B_r|} \int_{B_{2r} \backslash B_r} \Delta u\,dx > c(n) r^{q - 2}.$$ 
Using Jensen's inequality and summing over dyadic annuli, we obtain
$$\int_{B_1} F(\Delta u / \lambda)\,dx \geq \sum_{k = 1}^{\infty} |B_{2^{-k}} \backslash B_{2^{-k-1}}|F\left(c(n) 2^{-k(q-2)}/\lambda\right).$$
In particular, the Orlicz norm $\|\Delta u\|_{L^F(B_1)} = \infty$ if
$$\int_{1}^{\infty} t^{-\frac{n}{2-q}}F(t)\,\frac{dt}{t} = \infty.$$
Examples $F(t)$ that agree with $t^{\frac{n}{2-q}}|\log t|^{-p}$ for $t$ large and $0 \leq p \leq 1$ satisfy this condition, and give weaker norms 
than $L^{\frac{n}{2-q}}$. (For a reference on Orlicz spaces, see e.g. \cite{Sim}).
\end{rem}

\subsection{Estimate without Convex Analysis}

The following estimate is a certain analogue of Lemma \ref{AnnulusMass}, in the general setting.

\begin{lem}\label{AnnulusDichotomy}
Let $w$ be a nonnegative function on $B_2 \subset \mathbb{R}^n$ with $w(0) = 0$, and let $p > \frac{n}{2}$. Then there exists $c_0 > 0$ depending
on $n,\,p$ such that for all $\epsilon \in (0,\,1)$, there exists some $\delta (\epsilon,\,n,\,p)$ such that either
$$ \left(\int_{B_2 \backslash B_{\epsilon}} |D^2w|^p \,dx\right)^{\frac{1}{p}} \geq \delta \sup_{\partial B_1} w,$$
or
$$\left(\int_{B_{2\epsilon}} |D^2w|^p \,dx\right)^{\frac{1}{p}} \geq c_0\epsilon^{\frac{n}{p}-2} \sup_{\partial B_1} w.$$
\end{lem}
\begin{proof}
After multiplying by a constant we may assume that $\sup_{\partial B_1} w = 1$.
Assume that the first case is not satisfied. Then by the Sobolev-Poincar\'e and Morrey inequalities we have
$$\|w - l\|_{L^{\infty}(B_2 \backslash B_{\epsilon})} < C(n,\,p,\,\epsilon)\, \delta$$
for some linear function $l$. Take $\delta$ so small that the right side is less than $\frac{1}{8}$. 

By the hypotheses on $w$, we have $l(0) > 1/2$. Indeed, after a rotation we have $l(e_n) > \frac{7}{8}$ and $l(-2e_n) \geq -\frac{1}{8}$.  

Let $\tilde{w}(x) = (w - l)(\epsilon x)$. Then $|\tilde{w}| < \frac{1}{8}$ in $B_2 \backslash B_1$, and furthermore, $\tilde{w}(0) < -\frac{1}{2}$.
It follows again from standard embeddings that 
$$\left(\int_{B_2} |D^2\tilde{w}|^p \,dx\right)^{\frac{1}{p}} > c_0(n,\,p).$$ 
Scaling back, we obtain the desired inequality.
\end{proof}

\section{Proof of Theorem \ref{RealMain}}\label{RealMainProof}
We recall some estimates on the geometry of solutions to $\det D^2u \geq 1$. The first says that the volume of sub-level sets grows at most as fast as for the paraboloids with Hessian determinant $1$:
\begin{lem}\label{SectionGrowth}
Assume that $\det D^2u \geq 1$ in a convex subset of $\mathbb{R}^n$ containing $0$, with $u \geq 0$ and $u(0) = 0$. Then 
$$|\{u < h\}| < C(n)h^{n/2}$$
for all $h > 0$.
\end{lem}
The proof follows from the affine invariance of the Monge-Amp\`{e}re equation and a quadratic barrier (see e.g. \cite{M}, Lemma $2.2$).

Using Lemma \ref{SectionGrowth} we can quantify how quickly $u$ grows from a singularity. Below we fix $n \geq 3$ and $0 < k < \frac{n}{2}$, and we write $(x,\,y) \in \mathbb{R}^n$ with $x \in \mathbb{R}^{n-k}$ and $y \in \mathbb{R}^k$.
\begin{lem}\label{SingularityGrowth}
Assume that $\det D^2u \geq 1$ in $\{|x| < 1\} \cap \{|y| < 1\} \subset \mathbb{R}^n$, with $u \geq 0$ and $u = 0$ on $\{x = 0\}$. Then for all $r < 1$ we have
$$\inf_{y} \sup_{|x| = r} u(x,\,y) > c(n)r^{2-\frac{2k}{n}}.$$
\end{lem}
\begin{proof}
Take $c = c(n)$ small and assume by way of contradiction that for some $r_0$ the conclusion is false. Define $h = cr_0^{2-\frac{2k}{n}}$.
Then for some $y_0$ we have
$$\{y = y_0\} \cap \left\{|x| < (c^{-1}h)^{\frac{n}{2}\frac{1}{n-k}}\right\} \subset \{u < h\}.$$
Since $\{u < h\}$ is convex, it contains the convex hull of the set on the left and $\pm e_n$. We conclude that
$$|\{u < h\}| \geq \tilde{c}(n)(c^{-1}h)^{n/2},$$
which contradicts Lemma \ref{SectionGrowth} for $c$ small.
\end{proof}

The main theorem follows from the growth established in Lemma \ref{SingularityGrowth} and the convex analysis estimate Lemma \ref{CriticalBlowup}.
\begin{proof}[{\bf Proof of Theorem \ref{RealMain}:}]
Assume that $u$ agrees with a tangent plane on a set of dimension $k$. 
After subtracting the tangent plane, translating and rescaling we may assume that $u \geq 0$ on $\{|x| < 1\} \cap \{|y| < 1\}$, and that $u = 0$ on $\{x = 0\}$.  By Lemma \ref{SingularityGrowth}, we also have that
$$\inf_{\{|y| < 1\}} \sup_{|x| = r} u(x,\,y) \geq c(n)r^{2 - \frac{2k}{n}}.$$
Apply Lemma \ref{CriticalBlowup} on the slices $\{y = const.\}$ (taking $q = 2-\frac{2k}{n}$ and replacing $n$ by $n-k$) and integrate in $y$ to conclude that
$$\int_{\{|y| < 1\} \cap \{r < |x| < 1\}} (\Delta u)^{\frac{n}{2k}(n-k)} \,dx \geq c(n,\,k)|\log r|.$$
Taking $r \rightarrow 0$ completes the proof.
\end{proof}

\begin{rem}
By Remark \ref{OrliczGeneral}, one obtains the same result if $\Delta u$ is in the (weaker) Orlicz space $L^F$ for any convex $F: [0,\,\infty) \rightarrow 
[0,\,\infty)$ satisfying $F(0) = 0$ and
\begin{equation}\label{OrliczCondition}
 \int_{1}^{\infty} t^{-\frac{n(n-k)}{2k}}F(t)\, \frac{dt}{t} = \infty.
\end{equation}
\end{rem}

\section{Proof of Theorem \ref{ComplexMain}}\label{ComplexMainProof}
We first prove an analogue of Lemma \ref{SingularityGrowth}. We fix $n \geq 2$ and $0 < k < n$, and we use coordinates $(z,\,w) \in \mathbb{C}^n$ with $z \in \mathbb{C}^{n-k}$ and $w \in \mathbb{C}^k$.

\begin{lem}\label{ComplexSingularityGrowth}
Assume that $\det (\partial \dbar u) \geq 1$ in $\{|z| < 1\} \cap \{|w| < 1\} \subset \mathbb{C}^n$, with $u \geq 0$ and $u = 0$ on $\{z = 0\}$. Then for all $r < 1$ we have
$$\sup_{|w| < 1/4} \sup_{|z| = r} u(z,\,w) \geq c(n)r^{2 - \frac{2k}{n}}.$$
\end{lem}
\begin{proof}
Take $c = c(n)$ small and assume by way of contradiction that for some $r_0$ the conclusion is false. Let $h = cr_0^{2-\frac{2k}{n}}$. Then we have
$$\{|w| < 1/4\} \cap \{|z| < (c^{-1}h)^{\frac{n}{2(n-k)}}\} \subset \{u < h\}.$$
Here we used the plurisubharmonicity of $u$. (Note that the volume of the set on the left is much larger than $h^{n}$ for $c$ small.) 
The proof then proceeds as in the real case. For $c$ small, the convex quadratics $Q_t = 2h\left(16 |w|^2 + (c^{-1}h)^{-\frac{n}{n-k}}|z|^2\right) + t$ are supersolutions
that lie strictly above $u$ on $\partial(\{|w| < 1/4\} \cap \{|z| < r_0\})$ for $t \geq 0$. For some $t \geq 0$, $Q_t$ touches $u$ from above somewhere
inside this set, contradicting the maximum principle.
\end{proof}

\begin{proof}[{\bf Proof of Theorem \ref{ComplexMain}}]
Assume that the minimum set of $u$ contains an analytic sub-variety of dimension $k$. After a biholomorphic transformation and a rescaling, we may assume that
$u \geq 0$ on $\{|z| < 1\} \cap \{|w| < 1\}$ and $u = 0$ on $\{z = 0\}$ (see e.g. \cite{DD}, Theorem $32$ for details), and that
$$\|u\|_{W^{2, \frac{n(n-k)}{k}}(\{|w| < 1\} \cap \{|z| < 1\})} = C_0 < \infty.$$
(Here we used elliptic theory: $\Delta u$ controls $D^2u$ in $L^p$ for $1 < p < \infty$).

For any $r < 1/2$ we define 
$$u_r(z,\,w) = \frac{1}{r^{2-\frac{2k}{n}}}u(rz,\,w).$$ 
We claim that there exist $\epsilon,\, \delta > 0$ small depending on $n,\,k,\,C_0$ (but not $r$) such that
\begin{equation}
\int_{\{|w| < 1\} \cap \{\epsilon < |z| < 2\}} |D_z^2u_r|^{\frac{n(n-k)}{k}} |dz||dw| > \delta.
\end{equation}
Here $D_z^2$ denotes the Hessian in the $z$ variable.
We first indicate how to complete the proof given the claim. The invariance of this norm under the rescalings used to obtain $u_r$ gives that
$$\int_{\{|w| < 1\} \cap \{(\epsilon/2) r < |z| < r\}} |D^2 u|^{\frac{n(n-k)}{k}} |dz||dw| > \delta,$$
for all $r < 1$.
By summing this over the annuli $\{|w| < 1\} \cap \{(\epsilon/2)^{k+1} < |z| < (\epsilon/2)^{k}\}$ we eventually contradict the upper bound on the $W^{2,\,\frac{n(n-k)}{k}}$ norm of $u$.

We now prove the claim. By Lemma \ref{ComplexSingularityGrowth}, there exists some $(z_0,\,w_0) \in \{|z| = 1\} \cap \{|w| < 1/4\}$ with $u_r(z_0,\,w_0) \geq c(n) > 0$. Let
$$M(w) = u_r(z_0,\,w).$$
Since $M(w)$ is positive and subharmonic, we have by the mean value inequality that 
\begin{equation}\label{MVI}
\int_{\{|w| < 1\}} M(w)|dw| > c(n) > 0.
\end{equation}
By Lemma \ref{AnnulusDichotomy}, for all $\epsilon$ small, there exists $\delta(n,\,k,\,\epsilon)$ such that either
$$\left(\int_{\{\epsilon < |z| < 2\}} |D_z^2u_r(z,\,w)|^{\frac{n(n-k)}{k}} \,|dz|\right)^{\frac{k}{n(n-k)}} \geq \delta M(w)$$
or
$$\left(\int_{\{|z| < 2\}} |D_z^2u_r(z,\,w)|^{\frac{n(n-k)}{k}} \,|dz|\right)^{\frac{k}{n(n-k)}} \geq c(n,\,k) \epsilon^{-\frac{2(n-k)}{n}} M(w).$$
Let $A_{\epsilon}$ be the set of $w$ such that the first case holds. We conclude from the scale-invariance of the norm we consider that
\begin{align*}
C_0 &\geq c(n,\,k) \int_{A_{\epsilon}^c} \left(\int_{\{|z| < 2\}} |D_z^2u_r|^{\frac{n(n-k)}{k}}\,|dz| \right)^{\frac{k}{n(n-k)}}\,|dw| \\
&\geq c(n,\,k) \epsilon^{-\frac{2(n-k)}{n}}\int_{A_{\epsilon}^c} M(w) |dw|.
\end{align*}
By taking $\epsilon(n,\,k,\,C_0)$ small, we conclude that the mass of $M(w)$ in $A_{\epsilon}^c$ is less than its mass in $A_{\epsilon}$. We conclude from the estimate (\ref{MVI}) that
\begin{align*}
\left(\int_{\{|w| < 1\} \cap \{\epsilon < |z| < 2\}} |D_z^2u_r|^{\frac{n(n-k)}{k}} |dz||dw|\right)^{\frac{k}{n(n-k)}} &\geq c(n,\,k) \delta \int_{A_{\epsilon}} M(w) \,|dw| \\
&\geq c(n,\,k)\delta(n,\,k,\, C_0),
\end{align*}
completing the proof.
\end{proof}

\begin{rem}\label{Approximation}
To emphasize ideas we assumed $u$ has enough (qualitative) regularity to perform the above computations. This can be justified by standard approximation methods using mollifications
$u_{\epsilon}$ of $u$. The key points are that $u_{\epsilon}$ solve $\det^{1/n} (\partial \dbar u_{\epsilon}) \geq 1$ by the concavity of $\det^{1/n}$, approximate $u$ in $W^{2,\,p}$, and go to zero on $\{z = 0\}$ locally uniformly
in $\epsilon$ by the upper semicontinuity of plurisubharmonic functions.
\end{rem}

\begin{rem} 
Theorem~\ref{ComplexMain} actually implies a slightly more general result.  Namely, if $u$ is plurisubharmonic on $B_1$ and satisfies $\det \partial\dbar u \geq 1$, and $\Delta u\in L^{p}$ for some $p\geq \frac{n}{k}(n-k)$, then $u$ cannot be pluriharmonic when restricted to any analytic set of dimension greater than or equal to $k$.  This follows from Theorem~\ref{ComplexMain} and the proof of Theorem 35 in \cite{DD}.
\end{rem}

\section{Applications}\label{Applications}
As a consequence of Theorem \ref{RealMain} we obtain interior estimates for the real Monge-Amp\`{e}re equation depending on the $W^{2,\,p}$ norm of the solution, for 
any $p \geq \frac{n(n-1)}{2}$. This extends a result of Urbas \cite{U} to the equality case $p = \frac{n(n-1)}{2}$.

\begin{rem}
 In fact, we obtain interior estimates depending on certain Orlicz norms that are slightly weaker than $L^{\frac{n(n-1)}{2}}$.
\end{rem}

We recall the definition of sections of a convex function. Let $u$ be a convex function on $B_1 \subset \mathbb{R}^n$. If $l$ is a supporting linear function to $u$ at $x \in B_1$, we set
$$S_h^l(x) = \{u < l + h\} \cap B_1.$$

\begin{lem}\label{UniversalSectionHeight}
Assume that $\det D^2u \geq 1$ in $B_1 \subset \mathbb{R}^n$, and that $\|u\|_{W^{2,\,p}(B_1)} < C_0$ for some $p \geq \frac{n(n-1)}{2}$. Then there exists $h_0 > 0$ depending only
on $n,\,p$ and $C_0$ such that
$$S_{h_0}^l(x) \subset \subset B_1$$
for all $x \in B_{1/2}$ and supporting linear functions $l$ at $x$.
\end{lem}
\begin{proof}
The result follows from a standard compactness argument using the closedness of the condition $\det D^2 u \geq 1$ under uniform convergence, 
the lower semicontinuity of the $W^{2,\,p}$ norm under weak convergence and Theorem \ref{RealMain}.
\end{proof}

\begin{rem}\label{OrliczDependence}
 The conclusion is the same if the Orlicz norm $\|\Delta u\|_{L^F(B_1)} < C_0$ for some $F$ satisfying condition 
 (\ref{OrliczCondition}) for $k = 1$, and in addition e.g. $\|u\|_{W^{2,\,2}(B_1)} < C_0$.
 The argument is by compactness again, but one has to work harder to extract a limit whose Hessian has bounded Orlicz norm.
 Rather than using weak $W^{2,\,2}$ convergence of a subsequence $\{u_k\}$, invoke the Banach-Saks theorem and use the strong convergence in $W^{2,\,2}$ of Ces\`{a}ro means 
 $\frac{1}{N}\sum_{k = 1}^N u_k$. The convexity of $F$ then implies that the Hessian of the limit has bounded Orlicz norm. 
 
 (In order to use Banach-Saks we need control of $\Delta u$ in $L^p$ for some $p > 1$, which does not follow from bounded Orlicz norm. This is the
 reason for the second condition).
 \end{rem}

Interior e.g. $C^{2,\alpha}$ estimates and $W^{2,\,q}$ estimates in terms of $\|u\|_{W^{2,\,\frac{n(n-1)}{2}}(B_1)}$ follow, where the estimates also depend on $n,\, \alpha$ and the $C^{\alpha}$ norm (resp. $n$, $q$ and the modulus of continuity)
of $\det D^2u$. 

Indeed, by Lemma \ref{UniversalSectionHeight} we have that $S_{h_0}^l(x) \subset \subset B_1$ for some
universal $h_0$ and all $x \in B_{1/2}$. Since $\det D^2u$ also has an upper bound (depending on the $C^{\alpha}$ norm or modulus of continuity of $\det D^2u$),
we have the lower volume bound $|S_{h_0}^l(x)| > ch_0^{n/2}$ for compactly contained sections (\cite{Ca}). Combining this with the diameter estimate $\text{diam.}(S_{h_0}^l(x)) < 2$, 
we see that the eigenvalues of the affine transformations normalizing these sections (taking $B_1$ to their John ellipsoids)
are bounded above and below by positive universal constants. The estimates follow by applying Caffarelli's results (see \cite{Ca1}) in the normalized sections, 
scaling back, and doing a covering argument.


\section*{acknowledgements}
C. Mooney would like to thank A. Figalli for comments and X. J. Wang for encouragement.  
T. Collins would like to thank the math department at Chalmers University, where the majority of this work was completed, for a perfect working environment.


\end{document}